\documentclass[12pt]{amsart}

\usepackage{enumerate}

\usepackage[active]{srcltx}
\usepackage{nicefrac}

\usepackage{amsthm,amssymb,setspace}
\usepackage[pagebackref,colorlinks,linkcolor=red,citecolor=blue,urlcolor=blue,hypertexnames=true]{hyperref}
\usepackage{amsrefs}
\usepackage[matrix, arrow]{xy}

\renewcommand{\deg}{{\rm deg}}

\newcommand{\A}{\mathcal A}

\newcommand{\N}{\mathbb N}
\newcommand{\R}{\mathbb R}

\newcommand{\C}{\mathbb C}

\newcommand{\M}{\mathcal M}

\theoremstyle{plain}
\newtheorem*{theorem*}{Theorem}
\newtheorem{theorem}{Theorem}[section]
\newtheorem{corollary}[theorem]{Corollary}
\newtheorem{lemma}[theorem]{Lemma}
\newtheorem{proposition}[theorem]{Proposition}

\theoremstyle{definition}
\newtheorem*{definition*}{Definition}
\newtheorem{definition}[theorem]{Definition}

\theoremstyle{remark}

\newtheorem{remark}[theorem]{Remark}
\newtheorem{example}[theorem]{Example}

\begin{document}

\onehalfspace

\title{Polynomials with and  without determinantal representations}

\author{Tim Netzer}
\address{Tim Netzer, Universit\"at Leipzig, Germany}
\email{netzer@math.uni-leipzig.de}

\author{Andreas Thom}
\address{Andreas Thom, Universit\"at Leipzig, Germany}
\email{thom@math.uni-leipzig.de}

\begin{abstract}  The problem of writing real zero polynomials as determinants of linear matrix polynomials has recently attracted a lot of attention. Helton and Vinnikov \cite{hevi} have proved that any real zero polynomial in two variables has a determinantal representation. Br\"and\'{e}n \cite{bran} has shown that the result does not extend to arbitrary numbers of variables, disproving the generalized Lax conjecture. We prove that in fact {\it almost no} real zero polynomial admits a determinantal representation; there are dimensional differences between the two sets. So the generalized Lax conjecture fails badly. The result follows from a general upper bound on the size of linear matrix polynomials. We then provide a large class of surprisingly simple explicit real zero polynomials that do not have a determinantal representation, improving upon Br\"and\'{e}n's mostly unconstructive result.  We finally characterize polynomials of which some power has a determinantal representation, in terms of an algebra with involution having a finite dimensional representation. We use the characterization to prove that any quadratic real zero polynomial has a determinantal representation, after taking a high enough power. Taking powers is thereby really necessary in general. The representations emerge explicitly, and we characterize  them up to unitary equivalence.\end{abstract}

\maketitle

\tableofcontents

\section{Introduction} A \textit{(hermitian) linear matrix polynomial} (or  \textit{matrix pencil}) $\M$ is an expression of the following form: $$  M_0+x_1M_1+\cdots +x_nM_n,$$ where each $M_i\in{\rm{H}}_k(\C)$ is a complex hermitian $k\times k$-matrix, and $x_1,\ldots,x_n$ are variables. Equivalently, $\M$ can be viewed as a hermitian matrix with  linear polynomials as its entries. We refer to $k$ as the \textit{size} of $\M$. In the special case that all $M_i\in{\rm{Sym}}_k(\R)$ are real symmetric matrices, we call the matrix polynomial \textit{symmetric}.

 Linear matrix polynomials are of importance for example in polynomial optimization. Let $S(\M)$ denote the set of points where $\M$ is positive semidefinite: $$S(\M):=\left\{ a\in\R^n\mid \M(a)\succeq 0\right\}.$$   Such sets are called  \textit{spectrahedra}, and they are precisely the sets on which semidefinite programming can be performed. A great lot of problems from various branches of mathematics can be transformed to semidefinite programming problems, and there exist efficient methods to solve these problems. For more information we refer the reader to \cite{wo} and the references therein. 
  
It is not hard to see that if $0$ belongs to the interior of $S(\M)$, then it is definable  by a linear matrix polynomial (of possibly smaller size) with $M_0$ positive definite (see \cite{rago}, Section 1.4). After conjugation with a unitary matrix we find $M_0=I$, i.e. $\M$ is \textit{monic}.  We will always make this assumption in our work.

It is now an important problem to find out which sets are spectrahedra. Clearly spectrahedra are always convex and closed, and can be defined by simultaneous real polynomial inequalities.  
But there are more necessary conditions. Consider $p=\det\M$, the determinant of the linear matrix polynomial $\M$. It is a real polynomial, both in the hermitian and symmetric case. Now note that the spectrahedron $S(\M)$ can be retrieved from the polynomial $p$ only. It consists of those points $a$ for which $p$ does not have a zero between the origin and $a$ (see Remark \ref{rz} below): $$S(\M)=\left\{ a\in\R^n\mid p_a(t):=p(t\cdot a) \mbox{ has no roots in } [0,1)\right\}.$$ Helton and Vinnikov \cite{hevi} call the set on the right  the \textit{rigidly convex set} defined by $p$, and we will denote it by $S(p)$. 
Since $p$ arises as a determinant of a linear matrix polynomial, it has a  strong property. It fulfills  $p(0)=1$  and $$\forall a\in\R^n \quad p(\mu\cdot a)=0 \Rightarrow \mu\in\R.$$ The second property follows immediately from the fact that hermitian matrices have only real eigenvalues (see also again Remark \ref{rz} below). Polynomials with these two properties are called \textit{real zero polynomials}, or \textit{RZ-polynomials}, for short. 
We have now observed the first result of Helton and Vinnikov: each spectrahedron is of the form $S(p)$, for an RZ-polynomial $p$. This precludes for example a set like $$\{ (a,b)\in\R^2\mid 1-a^4-b^4\geq 0\}$$ from being a spectrahedron, as one easily checks.

A good approach to check whether a set is a spectrahedron is now to first realize it as $S(p)$ for some RZ-polynomial $p$, and then try to realize  $p$ as a determinant of a linear matrix polynomial. Note however that $S(p)$ could of course be a spectrahedron without $p$ being a determinant. It would for example be sufficient to represent some power $p^r$ as a determinant. Even representing a product $q\cdot p$ as a determinant would be enough, as long as $S(q\cdot p)=S(p)$. Also note that an RZ-polynomial can be the determinant of a hermitian linear matrix polynomial without being the determinant of a symmetric linear matrix polynomial. We will see examples of this fact below. 

Finally note that representing a polynomial as a determinant is \textit{always} possible, if one omits the condition that $M_0$ is positive semidefinite. This was proven by Helton, McCullough and Vinnikov \cite{hemcvi} and more elementary by Quarez \cite{qu}. Omitting the condition that each $M_i$ is hermitian makes the problem even simpler, as for example explained in \cite{qu} on page 7. The link to spectrahedra and semidefinite programming is then lost, however. So in our work, a linear matrix polynomial is always hermitian and monic.

Now Helton and Vinnikov \cite{hevi} prove the following remarkable result in the two-dimensional case:

\begin{theorem*} If $p\in\R[x,y]$ is an RZ-polynomial of degree $d$, then $p$ is the determinant of a symmetric linear matrix polynomial of size $d$.
\end{theorem*}
\noindent This shows that each rigidly convex set in $\R^2$ is a spectrahedron. 
As observed by Lewis, Parillo and Ramana \cite{lepara}, the theorem also solves the Lax conjecture, which was originally formulated in a homogenized setup, i.e. for so called \textit{hyperbolic} polynomials.
Helton and Vinnikov already  note that their result cannot  hold  as stated in higher dimensions. A count of parameters shows that there are much more RZ-polynomials of degree $d$ than could possibly be realized as determinants of symmetric linear matrix polynomials of size $d$. The same argument shows that also hermitian matrix polynomials of size $d$ are not enough. Helton and Vinnikov conjectured however that their result is true  if one allows for matrices of size larger than $d$.  Br\"and\'{e}n \cite{bran} has now recently disproved this conjecture. 

His work contains the following results: For a certain subclass of RZ-polynomials of degree $d$ he first proves that the existence of a determinantal representation   implies the existence of a representation of size $d$, both in the hermitian and symmetric case. Since the subclass is still large enough, a count of parameters then implies that many among these polynomials can not have a determinantal representation at all, if the number of variables is large. In a second section he then even produces an explicit RZ-polynomial for which \textit{no power} can have a determinantal representation. The example is constructed from a matroid (the V\'{a}mos cube) that cannot be realized by a subspace arrangement, since failing to fulfill the Ingleton inequalities. His polynomial has $8$ variables and is of degree $4$. 

Our contribution is the following. In Section \ref{size} we examine the possible size of a determinantal representation. We prove some upper  bounds in Theorems \ref{upbound} and \ref{cone}, and some lower bounds in Theorems \ref{boundreal} and \ref{bound}.  In Section \ref{norep} we deduce that almost no real zero polynomial admits a determinantal representation.  In fact there are dimensional differences between the set of real zero polynomials and the set of polynomials with a determinantal representation. This will follow by a count of parameters, using our general upper bound. So the generalized Lax conjecture fails badly.  We will then produce simple and explicit examples of polynomials without determinantal representations. This is in particular interesting, since there is only a single such example  so far (the one from Br\"and\'en's paper). Our examples include polynomials of high degree, compared to the number of variables, and vice versa polynomials in many variables and low degree. There are examples with $S(p)$ compact and non-compact.

 In Section \ref{algebra} we characterize polynomials of which some power has a determinantal representation. For this purpose we construct an algebra with involution, asssociated with the real zero polynomial.  We show that this algebra has a $*$-representation on a finite dimensional Hilbert space, if and only if some power of the polynomial admits a determinantal representation of small size (Theorem \ref{rep}).  Similar algebras have been used before by different authors (see e.g. \cite{he,ro,ch,pa}), in attempts to linearize forms and realize polynomials as minimal polynomials of matrix pencils. Their results relate to our problem, but do not take into account the desire for \textit{hermitian} representations. 
In Section \ref{quadratic} we use our characterization to prove that any quadratic RZ-polynomial admits a determinantal representation, after taking a high enough  power (Theorem \ref{quart}). This shows that any quadratic rigidly convex set is a spectrahedron. Our result also contains an explicit method to construct the  determinantal representations, in contrast to the previous results, which are mostly unconstructive. We finally determine the occuring representations up to unitary equivalence, in Theorem \ref{unique}.

\section{On the size of linear matrix polynomials}\label{size}

We start by proving some results that will be helpful throughout this work. The first and easy pšroposition turns out to be crucial for many of the following results.

\begin{proposition}\label{eigen} Let $\M=I+x_1M_1+\cdots +x_nM_n$ be a linear matrix polynomial and $p:=\det \M$ its determinant. Then for each $a\in\R^n,$ the nonzero eigenvalues of $a_1M_1+\cdots+ a_nM_n$ are in one to one correspondence with the zeros of the univariate polynomial $p_a(t):=p(t\cdot a)$, counting multiplicities. The correspondence is given by the rule $\lambda \mapsto -\frac{1}{\lambda}$.
\end{proposition}
\begin{proof} Fix $a\in\R^n$ and let $c_a$ denote the characteristic polynomial of the hermitian  matrix $a_1M_1+\cdots +a_nM_n$. For any $\lambda\neq 0$ we have \begin{align*} c_a(\lambda) &= \det\left( -\lambda I +a_1M_1+ \cdots + a_nM_n\right) \\ &= (-\lambda)^k p\left(\frac{a}{-\lambda}\right)  = (-\lambda)^k p_a\left(-\frac{1}{\lambda}\right).  \end{align*} We see that each nonzero eigenvalue $\lambda$ of $\M(a)$ gives rise to a zero of $p_a$ by the above defined rule. We also see that each zero of $p_a$ arises in this way, since $0$ is not such a zero. 
Taking the derivative with respect to $\lambda$ in the above equality we see that also the multiplicity of $\lambda$ as a zero of $c_a$ coincides with the multiplicity of $-\frac{1}{\lambda}$  as a zero of $p_a$.  
\end{proof}

\begin{remark}\label{rz} (i) We  see that no $p_a$ can have a complex zero. Any such zero would give rise to a complex eigenvalue of a hermitian matrix, which is impossible. This shows that the determinant $p$ is indeed a real zero polynomial.

(ii) We also see that
$\M(a)$ is positive semidefinite if and only if the polynomial $p_a$ has no zeros in the interval $[0,1)$. This proves $S(\M)=S(p)$, as mentioned in the introduction.
\end{remark}

\begin{corollary} \label{rank} Let $\M$ be a linear matrix polynomial and assume $p=\det\M$ is of degree $d$. Then each matrix in the real vector space $$V_\M:={\rm{span}}_\R\left\{M_1,\ldots,M_n\right\}$$ has rank at most $d$, and the generic linear combination has rank precisely $d$.
\end{corollary}
\begin{proof}The rank of any matrix $a_1M_1+ \cdots +a_nM_n$ is the number of its nonzero eigenvalues, which by Proposition \ref{eigen} correspond to the zeros of the univariate polynomial $p_a$.  Now each $p_a$ has degree at most $d$, and thus at most $d$ zeros. For all $a$ for which $p_a$ has degree precisely $d$, the matrix is of rank precisely $d$. This is true for the generic choice of $a$.\end{proof}

The following result gives a general upper bound on the size of determinantal representations.

\begin{theorem}\label{upbound}
 Let $p\in\R[x_1,\ldots,x_n]$ be an RZ-polynomial of degree $d$. If $p$ has a symmetric/hermitian determinantal representation, then it has a symmetric/hermitian determinantal representation  of size $nd$.
\end{theorem}
\begin{proof}
Assume $$p=\det\left( I+x_1M_1+\ldots +x_nM_n\right)$$ for some matrices $M_i\in{\rm H}_k(\C)$. Let $K_i\subseteq \C^k$ be the kernel of the linear map defined by $M_i$. By Corollary \ref{rank} we find $\dim_\C K_i\geq k-d$ for all $i$. An easy induction argument involving the dimension formula for subspaces yields $$\dim\left(K_1\cap\ldots\cap K_n\right)\geq k-nd.$$ So if $k>nd$ we can simultaneously split off a $k-nd$ block of zeros of each $M_i$, by conjugation with a unitary matrix. This produces a determinantal representation of $p$ of size $nd.$ The same argument works with symmetric matrices and an orthogonal base change.
\end{proof}

\begin{remark}
Note that  Theorem \ref{upbound} not only shows that there is always a relatively small determinantal representation, but in fact that \textit{each} determinantal representation is relatively small. Larger representations only arise as trivial extensions of small ones.\end{remark}

We now want to proof that there is always a determinantal representation of very small size, if the spectrahedron contains a full dimensional cone. We need the following proposition:

\begin{proposition}\label{red} Let $V\subseteq{\rm{H}}_k(\C)$ be an $\R$-subspace of hermitian matrices, such that all elements of $V$ have rank at most $d$. If $V$ contains a positive semidefinite matrix of rank $d$, then there is some unitary matrix $Q\in M_{k}(\C)$ such that $$Q^*VQ \subseteq\left\{ \left(\begin{array}{c|c}A & 0 \\\hline 0 & 0\end{array}\right)\mid A\in {\rm{H}}_d(\C)\right\}.$$  If $V\subseteq{\rm{Sym}}_k(\R)$, then $Q$ can be chosen real orthogonal.
\end{proposition}
\begin{proof}
After a unitary/orthogonal change of coordinates we can assume that $V$ contains a matrix $$ P'=\left(\begin{array}{c|c} P & 0 \\\hline 0 & 0\end{array}\right)$$ where $P$ is a positive definite matrix of size $d$. Let $A'$ be an arbitrary matrix from $V$ and write $$A'=\left(\begin{array}{c|c}A & B \\\hline B^* & C\end{array}\right)$$ We have to show $B=0$ and $C=0$.

 We know that $A'+\lambda P'$ has rank at most $d$ for all $\lambda\in\R$, and the upper left block of size $d$ in this matrix has arbitrary large eigenvalues, for $\lambda$ big enough.
Consider any quadratic submatrix of size $d+1$ of $A'+\lambda P'$, containing this upper left block, obtained by deleting the same set of rows and columns:

$$ \left(\begin{array}{c|c}A+\lambda P & b \\\hline b^* & c\end{array}\right)$$ Here $b\in\C^d$ is a certain column of $B$, and $c$ is the corresponding diagonal entry of $C$. From the rank condition we see that the last column in this matrix is a linear combination of the first $d$ columns, at least for $\lambda\in\R$ big enough. If $v=(v_1,\ldots,v_d)^t$ is the vector of coefficients of this linear combination, we have $$ (A+\lambda P)v=b \mbox{ and } b^*v=c,$$ which implies $v^*(A+\lambda P)v=\overline{c}.$ This means that for large values of $\lambda$, the norm of $v$ must be arbitrary small. But his is only compatible with the condition $b^*v=c$ if $c=0$. Since $A+\lambda P$ is positive definite, this then implies $v=0$, and thus $b=0$.
We have now shown $B=0$, and this implies $C=0$, using again the rank condition for large values of $\lambda$.
\end{proof}

\noindent
Br\"and\'en has shown that an RZ-polynomial that arises as a shift of a hyperbolic polynomial always admits a very small determinantal representation, if it admits any at all. For such polynomials, the set $S(p)$ always contains a full dimensional cone. So the following is a generalization of  Theorem 2.2 from \cite{bran}:

\begin{theorem}\label{cone} Let $\mathcal{M}$ be a hermitian/symmetric linear matrix polynomial and let $d$ denote the degree of $p=\det\mathcal{M}$. If the spectrahedron defined by $\mathcal{M}$ contains a full dimensional cone, then $p$ can be realized as the determinant of a hermitian/symmetric matrix polynomial of size $d$.
\end{theorem}
\begin{proof}
If the whole positive half-ray through some $a\in\R^n$ is contained in the spectrahedron, then $a_1M_1+\cdots+a_nM_n$ is positive semidefinite. Since the generic linear combination  has rank $d$, there is such $a$ for which $a_1M_1+\cdots+a_nM_n$ has rank $d$. Now apply Proposition \ref{red} to reduce the size of $\mathcal{M}$ to $d$, without changing the determinant. 
\end{proof}

\begin{remark}
Note that we have shown in the proof of Theorem \ref{cone} that indeed \textit{any} large determinantal representation of $p$ arises as a trivial extension of a small one. This means that there can be no degree canceling when computing the determinant, except for trivial reasons.
\end{remark}

After proving upper bounds on the size of a linear matrix polynomial, we prove some lower bounds as well. 
We use results on spaces of symmetric matrices of low rank, of which there are plenty in the literature.
The following is the main result from Meshulam \cite{mes}, stated in the terminology of Loewy and Radwan \cite{lora}.

\begin{theorem}\label{meshu} Let $V\subseteq \rm{Sym}_k(\R)$ be a linear subspace such that all elements of $V$ have rank at most $d$. Then $$\dim V \leq \alpha(k,d)$$ which computes as follows.  For $d=2e$ even: $$ \alpha(k,d) = \left\{\begin{array}{ll} \binom{d+1}{2}  & \mbox{ if } 2k\leq 5e+1  \\  \binom{e+1}{2} +  e(k -e)  & \mbox{ if } 2k> 5e+1. \end{array}\right. $$ For $d=2e+1$ odd: $$ \alpha(k,d) = \left\{\begin{array}{ll} \binom{d+1}{2}  & \mbox{ if } 2k\leq 5(e+1)  \\  \binom{e+1}{2} +  e(k -e) +1  & \mbox{ if } 2k> 5(e+1). \end{array}\right. $$ 
\end{theorem}

To be able to apply this result, we note the following easy and probably well-known fact:
\begin{lemma}\label{dimension} Let $\M=I+x_1M_1+\cdots +x_nM_n$ be a  linear matrix polynomial.
If $S(\M)$ does not contain a full line, then $M_1,\ldots,M_n$ are $\R$-linearly independent. \end{lemma}
\begin{proof}Assume that  some $M_i$ is an $\R$-linear combination of the other $M_j$, and replace it by this linear combination. We  see that  $S(\M)$ is the inverse image under a linear map of some nonempty spectrahedron in $\R^{n-1}$. It thus contains a full line. 
\end{proof}

The following result now shows that under a mild compactness assumption, \textit{no} polynomial has a  very small determinantal representation, if the number of variables is large enough. This can also  be seen as a stricter version of the already explained count of parameters argument by Helton and Vinnikov.

\begin{theorem} \label{boundreal} Let $\mathcal{M}$ be a symmetric linear matrix polynomial of size $k$, defining a spectrahedron in $\R^n$ that does not contain a full line. Let $d$ denote the degree of $p=\det\mathcal{M}$ and assume  $n> \binom{d+1}{2}.$ 
 If $d$ is even  then  $$ k\geq \frac{2n}{d} +\frac{d-2}{4},$$ if $d$ is odd then $$ k\geq \frac{2(n-1)}{d-1} +\frac{d-3}{4}.$$ 
\end{theorem}
\begin{proof}  From Corollary \ref{rank}, Theorem \ref{meshu} and Lemma \ref{dimension} we obtain $n\leq\alpha(k,d)$. The result is now just a straightforward computation.
\end{proof}

\begin{remark}
Note that Lemma \ref{dimension} immediately implies $n\leq \binom{k+1}{2}$ in the setup of Theorem \ref{boundreal}. This however only gives a lower bound on $k$ depending on  the square root of $n$.
\end{remark}

\noindent
To obtain similar results for hermitian matrices we need some more preparation.

\begin{lemma}\label{herm} For $M\in{\rm{H}}_k(\C)$  write $M=R+iS$ with a real symmetric matrix $R$ and a real skew-symmetric matrix $S$. Define $$\widetilde{M}=\left(\begin{array}{c|c}R & S \\\hline -S & R\end{array}\right),$$ a real symmetric matrix of size $2k$. Then $\widetilde{M}$ has the same eigenvalues as $M$, with double multiplicities.\end{lemma}
\begin{proof} Let $\lambda$ be an eigenvalue of $M$ and $z\in\C^k$ a corresponding eigenvector. Writing $z=a+ib$ with $a,b\in\R^k$ this implies $$Ra-Sb=\lambda a \quad \mbox{and}\quad Rb+Sa=\lambda b.$$ So both $$\left(\begin{array}{c}-a \\b\end{array}\right)\quad \mbox{and}\quad \left(\begin{array}{c}b \\a\end{array}\right)$$ are eigenvectors with eigenvalue $\lambda$ of $\widetilde{M}$.
 Now let $z_1,\ldots,z_m\in\C^k$ be  complex vectors and write each $z_j=a_j+ib_j$ with $a_j,b_j\in\R^k$. One checks that $z_1,\ldots,z_m$ are $\C$-linearly independent  if and only if  the vectors $$\left(\begin{array}{c}-a_1 \\ b_1\end{array}\right),\left(\begin{array}{c}b_1 \\ a_1\end{array}\right),\ldots,\left(\begin{array}{c}-a_m \\ b_m\end{array}\right),\left(\begin{array}{c}b_m \\ a_m\end{array}\right)$$ are $\R$-linearly independent in $\R^{2k}$. This finishes the proof. \end{proof}

\begin{lemma} \label{redu}Let $\M$ be a hermitian linear matrix polynomial of size $k$, and write $\M=\mathcal{R}+i\mathcal{S}$ with real symmetric and skew-symmetric linear matrix polynomials $\mathcal{R}$ and $\mathcal{S}$. Define $$\widetilde{\M}:=\left(\begin{array}{c|c}\mathcal{R} & \mathcal{S} \\\hline -\mathcal{S} & \mathcal{R}\end{array}\right).$$ Then $\widetilde{\M}$ is a symmetric linear matrix polynomial of size $2k$ with $$\det\widetilde{\M}= (\det \M)^2.$$\end{lemma}

\begin{proof}Write $\widetilde{p}=\det\widetilde{\M}$ and $p=\det\M.$ By Lemma \ref{herm}, the eigenvalues of $\widetilde{\M}(a)$ are the same as the eigenvalues of $\M(a)$, just with double multiplicity, for each $a\in\R^n.$ Proposition \ref{eigen} implies that $\widetilde{p}_a$ has the same zeros as $p_a$, just with double multiplicities, for each $a\in\R^n$. So $\widetilde{p}_a=(p_a)^2$ for all $a$, which implies  $\widetilde{p}=p^2$.
\end{proof}
  We see that a spectrahedron can always be defined by a symmetric linear matrix polynomial. In fact, $\widetilde{\M}$  and $\M$ define the same spectrahedron, since their determinants are the same, up to a square. This fact was also observed in \cite{rago}, Section 1.4.

\noindent
From Lemma \ref{redu} we now immediately deduce the following analog of Theorem \ref{boundreal} for hermitian matrices:

\begin{theorem} \label{bound} Let $\mathcal{M}$ be a hermitian linear matrix polynomial of size $k$, defining a spectrahedron in $\R^n$ that does not contain a full line. Let $d$ denote the degree of $p=\det\mathcal{M}$ and assume  $n> \binom{2d+1}{2}.$ Then $$k\geq \frac{n}{2d}+\frac{d-1}{4}.$$
\end{theorem}

\begin{example}
Let $d=2$. Applying Theorem \ref{bound} shows that $n>10$ implies $k\geq \frac{n+1}{4}$. In the symmetric case, Theorem \ref{boundreal} gives a much stronger bound. If $n>3$, then $k\geq n$. So the RZ-polynomial $p_n=1-x_1^2-\cdots-x_n^2$ cannot be realized as the determinant of a symmetric linear matrix polynomial of size smaller than $n$, except possibly for $n=3$ (although the proof of Theorem \ref{high} below will show that also for $n=3$ there is no symmetric representation of size $2$).  It can indeed always be realized as the determinant of $$\left(\begin{array}{cccc}1 & x_1 & \cdots & x_n \\x_1 & 1 &  &  \\\vdots &  & \ddots &  \\x_n &  &  & 1\end{array}\right)$$ which is of size $k=n+1$. 
We will consider the case of quadratic polynomials in more detail in Section \ref{quadratic}. From the results we can for example produce the following two hermitian representations of $p_3$: $$\left(\begin{array}{cc}1+x_3 & x_1+ix_2 \\x_1-ix_2 & 1-x_3\end{array}\right) \mbox{ and } \left(\begin{array}{cc}1-x_3 & -x_1-ix_2 \\-x_1+ix_2 & 1+x_3\end{array}\right),$$ which are checked not to be unitarily equivalent. It will turn out that up to unitary equivalence, there are only these two representations of size two, and indeed already $p_4$ does not admit a $2\times 2$ hermitian representation any more (see Theorem \ref{unique}).
\end{example}

\section{Polynomials without determinantal representations}\label{norep}

The first result in this section is, that for suitable choices of $d$ and $n$, there are dimensional differences between the set $\mathcal{R}_{n,d}$ of real zero polynomials of degree $d$ in $n$ variables, and the set $\mathcal{D}_{n,d}$ of such polynomials with a determinantal representation. This is what we mean by saying that  \textit{almost no}Ê real zero polynomial has a determinantal representation. In the following, let $\R[x_1,\ldots,x_n]_d$ denote the finite dimensional vector space of polynomials of degree at most $d$. 

\begin{lemma} The set $\mathcal{R}_{n,d}\subseteq\R[x_1,\ldots,x_n]_d$ is a closed semialgebraic set of dimension $$\binom{d+n}{d}-1.$$
\end{lemma}
\begin{proof}
Being a real zero polynomial can be expressed in a formula of first order logic, using quantifiers. By quantifier elimination of the theory of real closed fields, the set $\mathcal{R}_{n,d}$ is a semialgebraic subset of $\R[x_1,\ldots,x_n]_d$. 
It is a well known fact that the zeros of a univariate polynomial depend continuously on the coefficients of the polynomial. It is thus easy to check that  $\mathcal{R}_{n,d}$ is closed.
Finally, Nuij \cite{nu} has proven that the set of all hyperbolic polynomials with only simple roots is open within the space of homogeneous polynomials. So $\mathcal{R}_{n,d}$ has nonempty interior in the subspace of $\R[x_1,\ldots,x_n]_d$ defined by the condition $p(0)=1$. This proves $\dim\mathcal{R}_{n,d}=\binom{d+n}{d}-1.$
\end{proof}

\begin{theorem}\label{det}
The set $\mathcal{D}_{n,d}\subseteq\R[x_1,\ldots,x_n]_d$ is a closed semialgebraic set of dimension at most $n^3d^2.$
\end{theorem}
\begin{proof}
Consider the semialgebraic mapping \begin{align*}\det\colon {\rm H}_{nd}(\C)^n &\rightarrow \R[x_1,\ldots,x_n]_{nd}\\ (M_1,\ldots,M_n)&\mapsto \det( I+x_1M_1+ \cdots +x_nM_n).\end{align*} The set $\mathcal{D}_{n,d}$ is the image of $\det$ intersected with $\R[x_1,\ldots,x_n]_d,$ by Theorem \ref{upbound}. So $\mathcal{D}_{n,d}$  is semialgebraic and of dimension at most $$\dim_\R  {\rm H}_{nd}(\C)^n =n^3d^2.$$ Now let $(p_j)_j$ be a sequence of polynomials from $\mathcal{D}_{n,d}$, converging to some polynomial $p\in\mathcal{R}_{n,d}$. Let $M_1^{(j)},\ldots,M_n^{(j)}$ be matrices of size $nd$ from a determinantal representation of $p_j$.
Since $S(p)$ contains some ball around the origin, and the degree of all $p_j$ is at most $d$, we can assume that each $S(p_j)$ contains some fixed ball around the origin. In view of Proposition \ref{eigen}, this means that the Eigenvalues and thus the norms of all $M_i^{(j)}$ are simultaneously bounded. So we can assume that each $M_i^{(j)}$ converges to some $M_i$. By continuity, this yields a determinantal representation of $p$. \end{proof}

Comparing the dimensions of $\mathcal{R}_{n,d}$ and $\mathcal{D}_{n,d}$ we get the following Corollary:

\begin{corollary}
For either  $d\geq 4$ fixed, and large enough values of $n$, or $n\geq 3$ fixed and large enough values of $d$, almost no polynomial in $\mathcal{R}_{n,d}$ has a determinantal representation. \end{corollary}

 Note that although there must exist many RZ-polynomials without determinantal representations, the above results are non-constructive. Beside Br\"and\'en's explicit polynomial constructed from the V\'{a}mos  cube,  there is a complete lack of examples. We want to close this gap by providing methods to produce many such explicit examples. 
 
 The first result can be understood as a strengthening of Br\"and\'{e}n's first result. He proves that among the shifted hyperbolic polynomials, there must be  many that do not have a determinantal representation. A little trick indeed even shows that \textit{none} of the considered RZ-polynomials has a representation. 
 
 \begin{theorem}\label{main} Let $p\in\R[x_1,\ldots,x_n]$ be an RZ-polynomial of degree $d,$  defining a rigidly convex set that contains a full dimensional cone, but  not a full line. If  $n> \binom{d+1}{2}$, then $p$ does not have a symmetric determinantal representation. If $n>d^2,$ then $p$ does not have a hermitian determinantal representation.

\end{theorem}
 \begin{proof}
 If $p$ had a determinantal representation, then it would have one of size $d$, by Theorem \ref{cone}. On the other hand, the matrices $M_1,\ldots,M_n$ occuring in such a  representation would be linearly independent, by Lemma \ref{dimension}. Comparing with the real dimension of the space of symmetric and hermitian matrices, we get $n\leq \binom{d+1}{2}$ in the symmetric case, and $n\leq d^2 $ in the hermitian case. This contradicts the assumption. \end{proof}

Br\"and\'{e}n has considered RZ-polynomials that arise as a shift of  hyperbolic polynomials. If $p\in\R[x_1,\ldots,x_n]$ is a degree $d$ RZ-polynomial, then its shifted homogenization is defined as $$\widetilde{p}:=(x_0+1)^d\cdot p\left(\frac{x}{x_0+1}\right).$$ This is again a RZ-polynomial of degree $d$, and its rigidly convex set contains a full dimensional cone. So these polynomials serve as a source of many examples.

\begin{example}\label{simple}
Consider $p_n=1-x_1^2 -\cdots -x_n^2$. For $n\geq 3$ we find that $$\widetilde{p}_n=(x_0+1)^2 -x_1^2 -\cdots -x_n^2$$ is not realizable as the determinant of a symmetric linear matrix polynomial. For $n\geq 4$ it is not realizable as a hermitian determinant. Note that for $n=3$  we can realize it as the determinant of the hermitian matrix $$\left(\begin{array}{cc}1+x_0+x_1 & x_2+ix_3 \\x_2-ix_3 & 1+x_0-x_1\end{array}\right) .$$ Splitting this matrix into a symmetric and a skew-symmetric part, and building the symmetric block matrix of size $4$ as explained in Lemma \ref{redu}, we get a symmetric determinantal representation of $\widetilde{p}_3^2$. We will show below that for any quadratic RZ-polynomial, a  high enough  power has a determinantal representation.\end{example}

\begin{example}\label{hyper}
We can apply Theorem \ref{main} also to polynomials that do not arise as a shifted homogenization. Consider for example the RZ-polynomial $$q_n=(x_1+\sqrt2)^2 -x_2^2 -\cdots -x_n^2-1,$$ whose zero set is a two-sheeted hyperboloid. For $n\geq 5$, it does not have a hermitian determinantal representation, for $n=4$ no symmetric determinantal representation.
\end{example}

The following result applies to cases where the degree is high, compared to the number of variables.

\begin{theorem}\label{high} Let $p\in\R[x_1,\ldots,x_n]$ be a real zero polynomial of degree $d,$ such that $S(p)$ does not contain a full line. Further suppose  $d\not\equiv 0,1,7\mod 8,$ and for each $a\in\R^n,$ the polynomial  $p_a$ has only simple zeros (including the zeros at infinity). If $n\geq 3$, then the shifted homogenization $\widetilde{p}$ does not have a symmetric determinantal representation. For $n\geq 4,$ it does not have a hermitian representation.
\end{theorem}
\begin{proof}
If $\widetilde{p}$ has a determinantal representation, then by applying Theorem \ref{cone} and dehomogenizing we see that $p$ has a representation of size $d$. Thus the real space $V$ spanned by the $n$ matrices occuring in such a representation contains only matrices with simple eigenvalues, by Proposition \ref{eigen}. The dimension of $V$ is $n$, by Lemma \ref{dimension}. This contradicts the main result of Friedland, Robbin and Sylvester \cite{frs}, Theorem B in the symmetric case, and Theorem D in the hermitian case.
\end{proof}

\begin{remark}
The work of Friedland, Robbin and Sylvester also contains results in the case that $d\equiv 0,1$ or $7\mod 8$, that are more technical. Although they can be used to obtain results in the spirit of Theorem \ref{high},  we decided not to include them, to keep the exposition more concise.
\end{remark}

\begin{example}
Consider again $p_n=1-x_1^2-\cdots-x_n^2.$  Theorem \ref{high} is another way to see that for $n\geq 3$, $\widetilde{p}_n$ does not have a symmetric representation, and  no hermitian one for $n\geq 4$. But we can now  rise the degree by for example considering $$p_{n,m}:=p_n(1+p_n)(2+p_n)\cdots(m+p_n).$$ If $n\geq 3$ and $m$ is not a multiple of $4$, then the shifted homogenization $\widetilde{p}_{n,m}$ does not have a symmetric determinantal representation. For $n\geq 4$ the same is true with hermitian representations. This contrasts the fact that taking high enough \textit{powers} of $p_n$ results in a polynomial whose shifted homogenization has a representation, as we will show in Section \ref{quadratic}.
\end{example}

So far, all counterexamples included the condition that $S(p)$ contains a full dimensional cone. We can also construct counterexamples with $S(p)$ compact, using Theorem \ref{det} again.
So let $\widetilde{p}$ be the shifted homogenization of a real zero polynomial  $p\in\mathcal{R}_{n,d}.$ 
Then $\widetilde{p}\in\mathcal{R}_{n+1,d}$ is again a real zero polynomial, and there are explicit such examples without a determinantal representation, as we have just shown. 

We now multiply $\widetilde{p}$ with a real zero polynomial defining a ball of radius $r>1$  around the point $(-1,0,\ldots,0)$: $$ q_r = \widetilde{p}\cdot \frac{r}{r-1}\left(1- \frac1r\left((x_0+1)^2 + x_1^2 +\cdots + x_n^2\right)\right).$$ Then $S(q_r)$ is clearly compact. Now if $q_r$ has a determinantal representation for \textit{some} $r>1$, it has a representation for \textit{all} $r>1$. This follows easily from the fact that $q_r$ and $q_s$ can be transformed to each other by shifting and scaling. 

Now for $r\to\infty$, the polynomials $q_r$ converge to $\widetilde{p},$ and in view of the closedness result from  Theorem \ref{det}, none of the $q_r$ can thus have a determinantal representation. Note that if no power of $\widetilde{p}$ has a determinantal representation, then no power of no $q_r$ can have  a determinantal representation, by the same argument.

\begin{example} Take $\widetilde{p}= (x_0+1)^2 -x_1^2 -x_2^2-x_3^2 -x_4^2.$ We find that $$2\left((x_0+1)^2-x_1^2-x_2^2-x_3^2-x_4^2\right)\left(1-\frac12\left((x_0+1)^2+x_1^2+x_2^2+x_3^2+x_4^2\right)\right)$$ does not have a determinantal representation.
\end{example}

\begin{example}
Let $\widetilde{p}\in\mathcal{R}_{8,4}$ be Br\"and\'{e}n's example, constructed from the V\'{a}mos cube. As above, by multiplying with a suitably shifted ball, we get a polynomial $q$ with $S(q)$ compact, and no power of $q$ has a determinantal representation.
\end{example}

\begin{remark}
Note that if we multiply \textit{any} $p\in\mathcal{R}_{n,d}\setminus\mathcal{D}_{n,d}$ with $1-\frac{1}{r}\left(x_1^2+x_1^2+\cdots +x_n^2\right)$, then for some large enough value of $r$, the result will be a polynomial without a determinantal representation, defining a compact set.
\end{remark}

\section{The generalized Clifford algebra associated with a real zero polynomial}\label{algebra}

We now consider the problem of representing some power of a real zero polynomial as a determinant. As explained in the introduction, we characterize this problem in terms of finite dimensional representations of an algebra with involution. We became aware that a similar approach has been used for the problem of linearizing forms, by Heerema \cite{he}, Roby \cite{ro} and Childs \cite{ch}, among others. A solution to their problem implies a determinantal representation for the polynomial, but not necessarily a hermitian one, and also  without the matrix $M_0$ being positive semidefinite.    Further, Pappacena \cite{pa} has used an algebra as below to realize polynomials as minimal polynomials of matrix pencils. From this one can also deduce  determinantal representations, this time even monic, i.e. with $M_0=I$, but still not necessarily with all other matrices being hermitian. We will see in Section \ref{quadratic} that the strive for hermitian representations needs some more work in general.

So let $p\in\R[x_1,\ldots,x_n]$ be a real zero polynomial of degree $d\geq 1$, and $$\hat{p}(x_0,\ldots,x_n):=x_0^d\cdot p\left( \frac{x}{x_0}\right)$$ its usual homogenization. In the free non-commutative algebra  $\C\langle z_1,\ldots,z_n\rangle$ consider the polynomial $$h_a:= \hat{p}(-a_1z_1-\cdots-a_nz_n, a_1,\ldots,a_n),$$ for $a\in\R^n.$ Let $J(p)$ be the two-sided ideal in $\C\langle z_1,\ldots,z_n\rangle$ generated by all the polynomials $h_a$,  with $a\in\R^n$. We equip $\C\langle z_1,\ldots,z_n\rangle$ with the involution defined by  $z_j^*=z_j,$ for all $j$. Then $J(p)$ is a $*$-ideal and we can define the involution on the quotient.

\begin{definition} We call the
 $*$-algebra $$\A(p):=\C\langle z_1,\ldots,z_n\rangle/ J(p) $$ the \textit{generalized Clifford algebra associated with $p$.}
\end{definition}

\begin{remark} Note that the ideal $J(p)$ is finitely generated, although we used infinitely many generators to define it. Write $$h_a=\sum_{\alpha\in\N^n, \vert \alpha\vert = d} a^{\alpha} q_{\alpha}$$ for suitable $q_{\alpha}\in\C\langle z_1,\ldots, z_n\rangle.$ It is then easy so check that the $q_{\alpha}$ generate the ideal $J(p)$.

\end{remark}

Under a (finite dimensional) unital $*$-representation of $\A(p)$ we will in the following understand an algebra homomorphism $\A(p)\rightarrow {\rm{M}}_k(\C)$ for some $k\in\N,$ preserving the unit and the involution. We call $k$ the \textit{dimension} of the representation.
The following is our main result in this section.

\begin{theorem}\label{rep} Let $p\in\R[x_1,\ldots,x_n]$ be a real zero polynomial of degree $d\geq 1$.  
\begin{itemize}\item[(i)] If  some power $p^r$ has a determinantal representation of size $rd$, then $\A(p)$ admits a  unital $*$-representation of dimension $rd$. \item[(ii)] If $p$ is irreducible and  $\A(p)$ admits a unital $*$-representation of dimension $k$, then $k=rd$ and $p^r$ has a  determinantal representation of size $rd$.   
\end{itemize}
\end{theorem}
\begin{proof}
For (i) assume $p^r=\det\left(I +x_1M_1+\cdots+ x_nM_n\right)$ for some hermitian matrices $M_j$ of size $rd$. Consider the unital $*$-algebra homomorphism $$\varphi\colon\C\langle z_1,\ldots,z_n\rangle\rightarrow {\rm{M}}_{rd}(\C), z_j\mapsto M_j.$$ For any $a\in\R^n$ we know by Proposition \ref{eigen} that the eigenvalues of $a_1M_1+\cdots+a_nM_n$ arise from the zeros of $p^r_a$ by the rule $\mu\mapsto -\frac{1}{\mu}$ (including possible zeros at infinity). These eigenvalues are precisely the zeros of the univariate polynomial $\hat{p}(-t,a_1,\ldots,a_n),$ so the minimal polynomial of $a_1M_1+\cdots +a_nM_n$ divides $\hat{p}(-t,a_1,\ldots,a_n)$. This means $\varphi(h_a)=0$, so $\varphi$ induces a representation of $\A(p)$ as desired.

For (ii) let $\varphi\colon\A(p)\rightarrow {\rm{M}}_{k}(\C)$ be a unital $*$-algebra homomorphism. Set $M_j:=\varphi\left( z_j +J(p)\right),$  consider the linear matrix polynomial $$\M=I+x_1M_1+\cdots +x_nM_n$$ and its determinant $q=\det\M$. From the defining relations of $\A(p)$ we know $$\hat{p}(-a_1M_1-\cdots-a_nM_n,a_1,\ldots,a_n)=0$$ for all $a\in\R^n$. So the eigenvalues of $a_1M_1+\cdots+a_nM_n$ are among $-\frac{1}{\mu},$ where $\mu$ runs through the zeros of $p_a$ (including possibly $\mu=\infty)$. Proposition \ref{eigen} implies that the zeros of $q$ are contained in the zeros of $p$, and also $\deg(q)=k$. Since every irreducible real zero polynomial defines a real ideal (which follows for example from \cite{bcr} Theorem 4.5.1(v)), the real Nullstellensatz implies that each irreducible factor of $q$ divides $p$. So $q$ divides some power of $p$, and since $p$ is itself irreducible, $q=p^r$ for some $r\geq 1.$ This now finally implies $k=rd$.
 \end{proof}

\begin{remark} One could of course also define the generalized Clifford algebra as a quotient of the free algebra over the real numbers, instead of the complex numbers as we did here. This would allow to characterize symmetric representations of  powers of $p$. But in view of Lemma \ref{redu}, that would only make sense when one is interested in determining  the \textit{lowest} possible power for which there exists a symmetric representation. Since the classification of algebras is often simpler over the complex numbers, we decided not to take this approach.
\end{remark}

\begin{example}\label{sphere}
Consider  $p_n=1-x_1^2-\cdots -x_n^2.$ We find  $\A(p_n)$ defined via the relations $$(a_1z_1+\cdots +a_nz_n)^2=\Vert a\Vert^2,$$ which is the classical Clifford Algebra ${\rm{Cl}}_n(\C)$. It is well known that ${\rm{Cl}}_n(\C) \cong {\rm{M}}_k(\C)$ for even $n$ and $k=2^{\frac{n}{2}},$ and ${\rm{Cl}}_n(\C)\cong {\rm{M}}_k(\C)\oplus {\rm{M}}_k(\C)$ for $n$ odd and $k= 2^{\frac{n-1}{2}}$.
  So ${\rm{Cl}}_n(\C)$ admits a $*$-algebra homomorphism to ${\rm{M}}_k(\C)$ with $k=2^{\lfloor\frac{n}{2}\rfloor},$ for any $n$. Thus the $2^{\lfloor\frac{n}{2}\rfloor-1}$-th power of $p_n$ has a determinantal representation of size $2^{\lfloor\frac{n}{2}\rfloor}$. In the case of $n=2m$ we can use the Brauer-Weyl matrices \cite{brwe} generating the Clifford Algebra. Let $$1:=\left(\begin{array}{cc}1 & 0 \\0 & 1\end{array}\right),1':=\left(\begin{array}{cc}1 & 0 \\0 & -1\end{array}\right),P:=\left(\begin{array}{cc}0 & 1 \\1 & 0\end{array}\right),Q:=\left(\begin{array}{cc}0 & i \\-i & 0\end{array}\right). $$

\noindent
Then consider the hermitian matrices $$ 1' \otimes \cdots \otimes 1'\otimes P\otimes 1\cdots \otimes 1$$ and $$ 1' \otimes \cdots \otimes 1'\otimes Q\otimes 1\cdots \otimes 1,$$ 
 where $\otimes$ denotes the Kronecker (tensor) product of matrices, the product is of length $m$, and both $P$ and $Q$ run through all $m$ possible positions in this product. The arising $2m=n$ matrices are hermitian and yield a determinantal representation of the $2^{m-1}$-th power of $p_n.$
In the case of $n$ odd one can use the additional matrix $1'\otimes \cdots \otimes 1'$ to construct a representation of ${\rm{Cl}}_n(\C). $ This yields for example $$\det \left(\begin{array}{cccc}1+x_5 & x_1+ix_3 & x_2+ix_4 & 0 \\x_1-ix_3 & 1-x_5 & 0 & -x_2-ix_4 \\x_2-ix_4 & 0 & 1-x_5& x_1+ix_3 \\0 & -x_2+ix_4 & x_1-ix_3 & 1+x_5\end{array}\right)=\left( 1-x_1^2-x_2^2-x_3^2-x_4^2 -x_5^2 \right)^2$$
Note that in the odd case there is another representation of ${\rm{Cl}}_n(\C)$, given by the respective negative matrices, which is not equivalent to the first one (in contrast to the even case, where these representations are unitarily equivalent).
\end{example}

\section{Quadratic polynomials}\label{quadratic}

In this section we construct a finite dimensional $*$-representation of $\A(p)$, if $p$ is quadratic. Note that already Pappacena \cite{pa} has proven $\A(p)$ to be isomorphic to the Clifford algebra in the quadratic case. We need to be more subtle, since we are looking for homomorphisms respecting the involution. We start with a lemma that was also noted by Pappacena, and include the proof for completeness.

\begin{lemma}\label{findim}
If $p\in\R[x_1,\ldots,x_n]$ is a quadratic real zero polynomial, then $\dim_\C \A(p)\leq 2^n.$
\end{lemma}
\begin{proof}  Let $V$ be the real subspace in $\A(p)$ spanned by the elements $z_i+J(p)$.
Each element $v\in V$ fulfills a real quadratic relation $v^2= rv+s$. For $v,w\in\A(p)$ write $v\equiv w$ if $v-w\in V.$ Clearly $v^2\equiv 0$ for all $v\in V$.
We compute $$0\equiv ((z_i+J(p)) +(z_j+J(p)))^2\equiv (z_iz_j+J(p)) + (z_jz_i+J(p)),$$ so $z_iz_j +J(p)\equiv -z_jz_i +J(p)$  holds in $\A(p)$. This proves that the elements $$z_{i_1}\cdots z_{i_r}+J(p)$$ with $i_1<\cdots<i_r$ generate $\A(p)$ as a vector space, which finishes the proof.
\end{proof}

 Write a quadratic real zero polynomial $p$ as $$p(x)=x^t A x +b^t x +1$$ with $A\in {\rm{Sym}}_n(\R)$ and $b\in\R^n$. Then $p_a(t)= a^tAa\cdot t^2 + b^ta\cdot t +1,$ and the condition that  $p_a$ has only real roots is $\frac14 a^tbb^ta -a^tAa\geq 0$.
 So $p$ being a real zero polynomial is equivalent to $$\frac14 bb^t-A\succeq 0,$$ and this matrix then has a positive symmetric square root.

When we use the Clifford Algebra ${\rm{Cl}}_n(\C)$ in the following, we denote its standard generators by $\sigma_1,\ldots,\sigma_n$. They fulfill the relations $$\sigma_j^2=1, \sigma_j^*=\sigma_j \mbox{ and } \sigma_j\sigma_i=-\sigma_i\sigma_j \mbox{ for } i\neq j.$$

\begin{proposition}\label{tocl} Let $p=x^tAx +b^tx +1\in\R[x_1,\ldots,x_n]$ be a quadratic real zero polynomial. Then there is a unital $*$-algebra homomorphism $$\A(p)\rightarrow {\rm{Cl}}_n(\C),$$ defined by the rule  $$ a_1z_1+\cdots+a_nz_n +J(p) \mapsto \sigma^t \left(\frac14 bb^t-A\right)^{\frac12}a  +\frac12 b^ta$$ for all $a\in\R^n$. If $\frac14 bb^t-A$ is invertible, this is an isomorphism.
\end{proposition}
\begin{proof}  We abbreviate $\left(\frac14 bb^t-A\right)^{\frac12}$ by $C$ and  $\sigma^t Ca  +\frac12 b^ta$ by $c_a$. We  denote the entries of the real symmetric matrix $Caa^tC$ by $q_{ij}$ and compute in ${\rm{Cl}}_n(\C):$ 
{\allowdisplaybreaks \begin{align*} c_a^2 &= \sigma^t C aa^tC\sigma + b^ta \sigma^t Ca +\frac14 (b^ta)^2 \\ &= \sum_{i,j} \sigma_i q_{ij}\sigma_j + b^ta\sigma^tCa +\frac14(b^ta)^2 \\ &= \sum_i q_{ii} + \sum_{i<j} (\underbrace{q_{ij}-q_{ji}}_{=0})\sigma_i\sigma_j +b^ta\sigma^tCa + \frac14 (b^ta)^2 \\ &= {\rm{tr}}(Caa^tC)+ b^ta\sigma^tCa + \frac14 (b^ta)^2\\ &= {\rm{tr}}(a^tC^2a) + b^ta\sigma^tCa + \frac14 (b^ta)^2\\ &=a^t\left(\frac14 bb^t-A \right)a + b^ta\sigma^tCa+\frac14 (b^ta)^2\\ &= \frac12 (b^ta)^2 +b^ta\sigma^tCa -a^tAa\\Ê&= b^ta\cdot  c_a -a^tAa.\end{align*} }
 Now  we define a unital $*$-algebra homomorphism $$\varphi\colon\C\langle z_1,\ldots,z_n\rangle \rightarrow {\rm{Cl}}_n(\C); \quad a_1z_1+\cdots+a_nz_n \mapsto c_a.$$  The ideal $J(p)$ is  in our case   generated by the polynomials $$h_a= (a_1z_1+\cdots +a_nz_n)^2 - b^ta\cdot (a_1z_1+\cdots+a_nz_n) +a^tAa,$$   
so $$\varphi(h_a)= c_a^2 - b^ta\cdot c_a + a^tAa=0.$$ Thus $\varphi$ is well defined on $\A(p)$. In case that $\frac14 bb^t-A$ is invertible, $\varphi$ is onto. So Lemma \ref{findim} finishes the proof, using that the vector space dimension of ${\rm{Cl}}_n(\C)$ is $2^n$.\end{proof}

Now we can prove the main result from this section.

\begin{theorem}\label{quart}
Let $p\in\R[x_1,\ldots,x_n]$ be a quadratic real zero polynomial. Then for $r=2^{\lfloor\frac{n}{2}\rfloor-1},$ $p^r$ has a (hermitian) determinantal representation of size $2^{\lfloor\frac{n}{2}\rfloor}.$
\end{theorem}
\begin{proof}
 We have seen in Proposition \ref{tocl} that there is a unital $*$-algebra homomorphism to ${\rm{Cl}}_n(\C)$. But as already described in Example \ref{sphere},  ${\rm{Cl}}_n(\C)$ admits a unital $*$-algebra homomorphism into ${\rm{M}}_{k}(\C)$, with $k=2^{\lfloor\frac{n}{2}\rfloor}$.  So we can apply Theorem \ref{rep} to finish the proof, noting that the case where $p$ is reducible is trivial.
\end{proof}

\begin{corollary}
Let $p\in\R[x_1,\ldots,x_n]$ be a quadratic real zero polynomial. Then $S(p)$ is a spectrahedron. 
\end{corollary}

\begin{remark}
We can compute the determinantal representations in the setup of Theorem \ref{quart} explicitly. Proposition \ref{tocl} gives an explicit morphism from $\A(p)$ to ${\rm{Cl}}_n(\C)$, and this yields an explicit representation in  ${\rm{M}}_k(\C),$ using for example the Brauer-Weyl matrices (see Examples \ref{ex1} and \ref{ex2} below).
\end{remark}

Also of interest is the question how many different representations for a real zero polynomial exist. Helton, Klep and McCullough \cite{hkmc} have for example characterized equivalent representations in terms of matricial spectrahedra, i.e. spectrahedra defined in $\left({\rm{Sym}}_k(\R)\right)^n,$ instead of $\R^n$ only.
Under a regularity condition on the polynomial, we see that the representations in Theorem \ref{quart} can be described completely, up to unitary equivalence.

\begin{theorem}\label{unique} Let $p=x^tAx +b^tx +1 \in\R[x_1,\ldots,x_n]$ be a quadratic real zero polynomial for which $\frac14 bb^t-A$ is invertible. Set $k=2^{\lfloor\frac{n}{2}\rfloor}.$ 

If $p^r$ has a determinantal representation of size $2r$, for some $r\geq 1$, then $r$ is a positive multiple of $\frac{k}{2}$. After a unitary change of variables, the representation splits into blocks of size $k$, each one representing $p^{\frac{k}{2}}$. 

If $n$ is even, then any two determinantal representations of $p^{\frac{k}{2}}$ of size $k$ are unitarily equivalent. If $n$ is odd then there are precisely two such representations, up to unitary equivalence.
\end{theorem}
\begin{proof} Note that the regularity condition implies that $p$ is irreducible.
Now let first $n$ be even.
From Proposition \ref{tocl} we know $\A(p)\cong {\rm{Cl}}_n(\C)\cong {\rm{M}}_k(\C)$.  A determinantal representation of $p^{r}$ of size $2r$ gives rise to a $*$-algebra representation of ${\rm{M}}_k(\C)$ of dimension $2r$. From the classification of $*$-subalgebras of matrix algebras we see that this representation splits into blocks, which are of size $k$ since ${\rm{M}}_k(\C)$ is simple. Finally, since every $*$-automorphism of a matrix algebra is conjugation with a unitary matrix, any two representations of size $k$ are unitarily equivalent.

Now let $n$ be odd. We have $\A(p)\cong{\rm{Cl}}_n(\C)\cong{\rm{M}}_k(\C)\oplus{\rm{M}}_k(\C),$ and this algebra has now precisely two irreducible $*$-representations up to unitary equivalence, both of size $k$. They are for example given by the Brauer-Weyl matrices and their negatives.
\end{proof}

We finish our work with two explicit examples for the above results.

\begin{example}\label{ex1} Consider $q_n= (x_1+\sqrt{2})^2-x_2^2 -\cdots -x_n^2-1$. Writing $q_n=x^tAx +b^ta +1$ we see $$\frac14 bb^t-A=I.$$ The above described homomorphism $\A(q_n)\rightarrow {\rm{Cl}}_n(\C)$ is given by the rule $$z_1+J(q_n)\mapsto \sigma_1 +\sqrt{2} $$ $$ z_j + J(q_n)\mapsto \sigma_j\ \mbox{Êfor }  j=2,\ldots,n.$$ We can substitute the Brauer-Weyl matrices (or their negatives) for the $\sigma_j$ and obtain one or two different representations, depending on whether $n$ is even or odd. Every other representation of some power is equivalent to a block sum of these minimal representations (and possibly trivial blocks, by Theorem \ref{cone}). An explicit example of a minimal representation is $$\det \left(\begin{array}{cccc}1+\sqrt{2}x_1 +x_5 & x_1+ix_3 & x_2+ix_4 & 0 \\x_1-ix_3 & 1+\sqrt{2}x_1 -x_5& 0 & -x_2-ix_4 \\x_2-ix_4 & 0 & 1+\sqrt{2}x_1 -x_5& x_1+ix_3 \\0 & -x_2+ix_4 & x_1-ix_3 & 1 +\sqrt{2}x_1 +x_5\end{array}\right)= q_5^2.$$
\end{example}

\begin{example}\label{ex2} Consider $\widetilde{p}_n= (x_0+1)^2 -x_1^2-\cdots-x_n^2$. Writing $\widetilde{p}_n= x^tAx +b^tx +1$ we see $$\frac14 bb^t-A= \left(\begin{array}{c|c}0 & 0 \\\hline 0 & I_n\end{array}\right),$$ and  the homomorphism $\A(\widetilde{p}_n)\rightarrow {\rm{Cl}}_{n+1}(\C)$ is given by the rule $$z_0\mapsto 1, z_j\mapsto\sigma_j \mbox{ for } j=1,\ldots,n.$$ As above this leads to representations, for example $$\det \left(\begin{array}{cccc}1+x_0 & x_1+ix_3 & x_2+ix_4 & 0 \\x_1-ix_3 & 1+x_0 & 0 & -x_2-ix_4 \\x_2-ix_4 & 0 & 1+x_0 & x_1+ix_3 \\0 & -x_2+ix_4 & x_1-ix_3 & 1+x_0\end{array}\right)=\widetilde{p}_4^2.$$
\end{example}

\section*{Acknowledgements}

We wish to thank Claus Scheiderer and Markus Schweighofer for valuable comments, especially on the algebra from Section \ref{algebra}.

\end{document}